\documentclass{article}
\usepackage{fullpage, amsfonts, colonequals, amsthm, amsmath, setspace,graphicx,tikz-cd}
\usepackage{mathrsfs}
\usepackage{enumitem,url}

\newtheorem{theorem}{Theorem}[section]

\newtheorem{lemma}[theorem]{Lemma}

\numberwithin{subcase}{case}
\numberwithin{claim}{theorem}
\numberwithin{conjecture}{section}

\newenvironment{definition}[1][Definition]{\begin{trivlist}
\item[\hskip \labelsep {\bfseries #1}]}{\end{trivlist}}

\newcommand{\ds}{\displaystyle}

\newcommand{\DQ}{\operatorname{DQ}}

\title{Luzin-type properties and the difference quotient of a real function.}
\author{Yury Andreev, Trevor J. Richards}

\begin{document}

\maketitle

\begin{abstract}
We introduce the notion of the difference quotient set of a real valued function $f$ on a set $E\subset[0,1]$, and compare this set to the range of $f$ on $E$.  We discuss the measure theoretic properties of both the range and the difference quotient set of $f$ over $E$ under different assumptions on $f$ and $E$.
\end{abstract}

\section{Introduction}

In the present paper we consider the regularity of the image of measurable sets under a continuous function.  An arbitrary function $f:[0,1]\to\mathbb{R}$ is said to possess the Luzin N-property if the image of any null set is a null set.  This property has been shown to have a number of interesting equivalent and related formulations (see~\cite{E} and citations there).  In particular, $f$ possesses the Luzin N-property if and only if $f(E)$ is measurable for each measurable subset $E\subset[0,1]$.  Closely related, and somewhat stronger, is the Banach S-property, which states that for any $\epsilon>0$, there is a $\delta>0$ such that if $E\subset[0,1]$ satisfies $\lambda(E)<\delta$, then $\lambda(f(E))<\epsilon$.  Both the Luzin N-property and the Banach S-property are analyzed in depth in~\cite{S}.

We are interested in what might be thought of as an inverse property to the Luzin N-property.  That is, we will find conditions on $f$ to ensure that the image of any positive mass set $E\subset[0,1]$ in turn has positive mass.  One could ask further whether there are reasonable conditions on $f$ which would ensure that the image of each positive mass set contains an interval.  The answer is negative. However if we consider instead the difference quotient of $f$ , which in effect introduces some interdependence to the function in question, then natural conditions on $f$ do exist which guarantee that the difference quotient set of $f$ over an arbitrary positive mass set will contain an interval.  To discuss this further we make the following definitions.

\vbox{
\begin{definition}  Let $g:[0,1]\to\mathbb{R}$ be a function, and let $F\subset[0,1]$ be a set.  Define $D=\{(x,x):0\leq x\leq1\}$.

\begin{itemize}
\item
For any $(x_1,x_2)\in([0,1]\times[0,1])\setminus D$, we define $\DQ_g(x_1,x_2)=\dfrac{g(x_2)-g(x_1)}{x_2-x_1}$.

\item
We define $$DQ_g(F)=\left\{DQ_g(x_1,x_2):(x_1,x_2)\in(F\times F)\setminus D\right\}.$$
\end{itemize}

\end{definition}
}

Throughout the remainder of the paper, we let a measurable function $f:[0,1]\to\mathbb{R}$ and a measurable set $E\subset[0,1]$ be fixed.  For $k\in\mathbb{R}$, we use the notation $\{f'=k\}$ to denote the set $\{x\in[0,1]:f'(x)=k\}$.  It is well known that absolutely continuous functions satisfy the Luzin N-property.  If $f$ is sufficiently smooth, then $\DQ_f$ will have the Luzin N-property as well.  If for example, $f\in C^1[0,1]$, then $\DQ_f$ will have continuous partial derivatives on $\left([0,1]\times[0,1]\right)\setminus D$, and will thus be locally Lipschitz on $\left([0,1]\times[0,1]\right)\setminus D$.  Locally Lipschitz functions of several variables have the Luzin N-property.  (The above facts are straight-forward exercises, the latter using the Cauchy--Schwartz inequality.)  Using the fact that if $E$ is null in $\mathbb{R}$, then $E\times E$ is null in $\mathbb{R}^2$, and thus $(E\times E)\setminus D$ is null as well, we obtain the following.


\begin{theorem}Assume that $\lambda(E)=0$.
\begin{itemize}
\item If $f$ is absolutely continuous then $\lambda\left(f(E)\right)=0$.

\item If $f\in C^1[0,1]$ then $\lambda\left(\DQ_f(E)\right)=0$.
\end{itemize}
\end{theorem}

As noted above, our main goal in this paper is to explore results which lie in the opposite direction.  In order to achieve these results we will need at least one of the following conditions to hold for $f$.

\begin{itemize}
\item $\mathscr{A}$: $\lambda\left(\left\{f'=k\right\}\right)=0$ for every $k\in\mathbb{R}$. 
\item $\mathscr{B}$: $\lambda\left(\left\{f'=0\right\}\right)=0$.
\item $\mathscr{C}$: $f$ is not linear on any positive mass subset of $[0,1]$.
\item $\mathscr{D}$: $f$ is not constant on any positive mass subset of $[0,1]$.

\end{itemize}

The dependence relations between the above concepts are as follows.

\begin{center}
\begin{tikzcd}[arrows=Rightarrow]
& &\mathscr{B} \arrow[drr] \\
\mathscr{A} \arrow[urr] \arrow[drr]& & & & \mathscr{D} \\
& &\mathscr{C} \arrow[urr]
\end{tikzcd}
\end{center}


An interesting consideration in the remainder of the paper is which of the assumptions above are necessary to achieve the desired conclusions.  In Section~\ref{sect: Proof of the main theorem.}, we will prove the following theorem.

\begin{theorem}\label{thm: Main.}
Assume that $\lambda(E)>0$ and that $f\in C^1[0,1]$.  Then the following hold.
\begin{enumerate}
\item If $f$ satisfies $\mathscr{B}$, then $\lambda\left(f(E)\right)>0$.
\item If $f$ satisfies $\mathscr{C}$, then $\lambda(\DQ_f(E))>0$.  In fact, for almost every $(x_1,x_2)\in (E\times E)\setminus D$, $\DQ_f(x_1,x_2)$ lies in the interior of $\DQ_f(E)$.
\end{enumerate}
\end{theorem}

The symmetry in the properties $\mathscr{A}-\mathscr{D}$ and in the two items of Theorem~\ref{thm: Main.} may lead one to suspect that the assumption $\mathscr{B}$ in Item~1 might be replaced by the weaker assumption $\mathscr{D}$.  However this is not the case.  In Section~\ref{sect: Limiting counter-example.} we will describe the construction of a function $g\in C^1[0,1]$ which satisfies $\mathscr{D}$, and find a positive mass subset of $[0,1]$ whose image under $g$ is null.

Regarding the second item of Theorem~\ref{thm: Main.} we have several remarks.  There is a well known measure theoretic fact that, for any positive mass set $G\subset\mathbb{R}$, the set $G-G=\{x-y:x,y\in G\}$ contains a neighborhood of the origin.  If $\lambda(E)>0$ and $f$ has a continuous derivative on $[0,1]$ and satisfies $\mathscr{B}$, then the first item of Theorem~\ref{thm: Main.} gives us that the set $f(E)-f(E)$ contains an interval centered at the origin.  While the set $f(E)-f(E)$ is certainly related to $\DQ_f(E)$, we have not been able to use this fact directly to prove that $\DQ_f(E)$ must contain an interval.  By instead making the assumption that $f$ satisfies $\mathscr{C}$, we establish in the second item of Theorem~\ref{thm: Main.} that the difference quotient set for $f$ on $E$ must contain an interval.

It is reasonable to ask if there is any version of the second item of Theorem~\ref{thm: Main.} which holds for the range of $f$ rather than for the difference quotient set.  However for any $n>0$ and any measurable function $g:[0,1]^n\to\mathbb{R}$, there is a set $F\subset[0,1]^n$ such that $\lambda(F)=1$ and $g(F)$ contains no interval.  We will include this construction in Section~\ref{sect: Limiting counter-example.} as well.

Finally, recall that $\DQ_f(E)$ is the set of values one obtains by taking difference quotients over all pairs in $(E\times E)\setminus D$, where $D$ is the diagonal, $D=\{(x,x):x\in[0,1]\}$.  It is worth noting that if one wished instead to consider the difference quotient of $f$ over an arbitrary positive mass set (even full mass set) $G\subset([0,1]\times[0,1])\setminus D$, the conclusion referred to in the preceding paragraph shows that the set $\{\DQ_f(x_1,x_2):(x_1,x_2)\in G\}$ would not necessarily contain an interval, regardless of what regularity conditions were imposed on $f$.

\section{Proof of Theorem~\ref{thm: Main.}.}\label{sect: Proof of the main theorem.}

Assume throughout this section that $\lambda(E)>0$ and that $f\in C^1[0,1]$.  To establish the first item of Theorem~\ref{thm: Main.}, assume further that $f$ satisfies $\mathscr{B}$, namely that $\lambda\left(\{f'=0\}\right)=0$.  We will begin with a definition.

\begin{definition}
Let $F,G\subset\mathbb{R}$ be measurable with $\lambda(G)>0$, and let $x\in\mathbb{R}$ be given.

\begin{itemize}
\item
We define the density of $F$ in $G$ to be $$\Delta(F,G)=\dfrac{\lambda(F\cap G)}{\lambda(G)}.$$

\item
We say that the density of $F$ at $x$ is $\Delta(F,x)=\ds\lim_{s\to0^+}\Delta(F,(x-s,x+s))$, if the limit exists.  If $\Delta(F,x)=1$, we say that $F$ has full density at $x$, and that $x$ is a density point of $F$.
\end{itemize}
\end{definition}

By the Lebesgue density theorem and the assumption $\mathscr{B}$, we can find some point $x_0\in E$ which is a density point of $E$, and at which $f'$ takes a non-zero value.

Define $m=f'(x_0)$, and assume that $m>0$, otherwise making the appropriate minor changes.  Since $f'$ is continuous, we can find some open interval $I\subset[0,1]$ with $x_0\in I$ and such that $f'(x)\geq m/2$ for all $x\in I$.  Since $x_0$ is a density point of $E$, we have $\lambda\left(E\cap I\right)>0$.  We will now show that $\lambda\left(f(E\cap I\right))>0$.  We will use that

$$\lambda\left(f(E\cap I)\right)=\ds\int_{f(E\cap I)}1dy.$$

Since $f'>0$ on $I$, $f$ is strictly increasing on $I$, so we may define $f^{-1}$ to be the branch of the inverse of $f$ which maps $f(I)$ back to $I$, and thus $f^{-1}(f(E\cap I))=E\cap I$.  In the integral above, we will use the substitution $x=f^{-1}(y)$, and thus we have $dy=f'(x)dx$.  We obtain

$$\lambda\left(f\left(E\cap I\right)\right)=\ds\int_{f(E\cap I)}1dy=\int_{E\cap I}f'(x)dx\geq\int_{E\cap I}\dfrac{m}{2}dx=\dfrac{m}{2}\lambda(E\cap I)>0,$$

which establishes the first item of Theorem~\ref{thm: Main.}.

To prove the second item of Theorem~\ref{thm: Main.}, we assume that $f$ satisfies $\mathscr{C}$, namely that $f$ is not linear on any positive mass subset of $[0,1]$.  We will begin with a definition and two lemmas.  The first lemma is the essential step, establishing the desired result in a single special case.  The second lemma then shows that in fact the setting of the first lemma is typical for the function $f$ which we are considering.

\begin{definition}
For a set $A\subset\mathbb{R}^2$ define the projections $$\Pi_x(A)=\{x:(x,y)\in A\text{ for some }y\in\mathbb{R}\},$$ and $$\Pi_y(A)=\{y:(x,y)\in A\text{ for some }x\in\mathbb{R}\}.$$
\end{definition}

\begin{lemma}\label{lem: Concrete context.}
Let $g$ be a function defined and continuously differentiable in a neighborhood of both $-1$ and $1$, and assume that $g(-1)=0=g(1)$, and that $g'(-1)\neq0\neq g'(1)$.  Let $F\subset\mathbb{R}$ denote some measurable set having full density at both $-1$ and $1$.  Then $\DQ_g(-1,1)=0$ is in the interior of $\DQ_g(F)$.
\end{lemma}

\begin{proof}

We will assume that $g'(-1),g'(1)>0$, otherwise making the appropriate minor changes.  Define $m=\min(g'(-1)/2,g'(1)/2)$ and $M=\max(2g'(-1),2g'(1))$.  Our first goal is to show that $0\in\DQ_g(F)$.  We begin by restricting $g$ to the union of two non-trivial closed intervals $I^-$ and $I^+$ (to be determined momentarily) containing $-1$ and $+1$ respectively.  We define $F^-=F\cap I^-$ and $F^+=F\cap I^-$.  $I^-$ and $I^+$ are any fixed non-empty closed intervals chosen small enough and in such a way that the following hold.

\begin{enumerate}
\item\label{item: I^-,I^+ bounded away from origin.}
$I^-\subset(-3/2,-1/2)$ and $I^+\subset(1/2,3/2)$.

\item\label{item: g continuous on I^- and I^+.}
$g$ is defined and continuously differentiable on $I^-\cup I^+$.

\item\label{item: Bound of g' on I^- and I^+.}
$m<g'<M$ on $I^\pm$ (which may be done since $g'$ is continuous close to $\pm1$).

\item\label{item: Density of F on I^- and I^+.}
The density of $F$ on each of $I^-$ and on $I^+$ is greater than $1-m/4M$ (so that $\Delta(I^\pm\setminus F^\pm,I^\pm)<m/4M$).

\item\label{item: g(I^-)=g(I^+).}
$g(I^-)=g(I^+)$ (which may be done because $g$ is strictly monotonic near both $-1$ and $+1$).

\end{enumerate}

Since $m<g'<M$ on $I^\pm$ (and thus on $F^\pm$) by Item~\ref{item: Bound of g' on I^- and I^+.} above, the same reasoning found in the proof of the first item of Theorem~\ref{thm: Main.} gives us $\lambda(g(I^\pm\setminus F^\pm))\leq M\lambda(I^\pm\setminus F^\pm)$ and $\lambda(g(I^\pm))\geq m\lambda(I^\pm)$.  Since $I^\pm\setminus F^\pm\subset I^\pm$, $g(I^\pm\setminus F^\pm)\subset g(I^\pm)$, so that
\begin{equation}\label{eqn: Density inequality.}
\Delta(g(I^\pm\setminus F^\pm),g(I^\pm))=\dfrac{\lambda(g(I^\pm\setminus F^\pm))}{\lambda(g(I^\pm))}\leq
\dfrac{M\lambda(I^\pm\setminus F^\pm)}{m\lambda(I^\pm)}.
\end{equation}

Moreover $\lambda(I^\pm\setminus F^\pm)=\Delta(I^\pm\setminus F^\pm,I^\pm)\cdot\lambda(I^\pm)$, so that Inequality~\ref{eqn: Density inequality.} and Item~\ref{item: Density of F on I^- and I^+.} above gives us $$\Delta(g(I^\pm\setminus F^\pm),g(I^\pm))\leq
\dfrac{M\Delta(I^\pm\setminus F^\pm,I^\pm)\cdot\lambda(I^\pm)}{m\cdot\lambda(I^\pm)}<\dfrac{M(m/4M)}{m}=\dfrac{1}{4}.$$

Therefore $\Delta(g(F^\pm),g(I^\pm))>3/4$.  Since $g(I^+)=g(I^-)$ (by Item~\ref{item: g(I^-)=g(I^+).} above), we conclude that $g(F^-)\cap g(F^+)\neq\emptyset$.  Choose some $x^-\in F^-$ and $x^+\in F^+$ (which are distinct points by Item~\ref{item: I^-,I^+ bounded away from origin.} above) such that $g(x^-)=g(x^+)$.  Then $0=\DQ_g(x^-,x^+)\in\DQ_g(F)$.

We now wish to show that $0$ is in fact in the interior of $\DQ_g(F)$.  First we define $\Gamma^-$ and $\Gamma^+$ to be the portions of the graph of $g$ restricted to $I^-$ and $I^+$ respectively.  We set $\Gamma=\Gamma^-\cup\Gamma^+$.  For $\theta\in\mathbb{R}$, let $R_\theta:\mathbb{R}^2\to\mathbb{R}^2$ denote the operation of counter-clockwise rotation around the origin through $\theta$ radians.  Define $\Gamma^-_\theta=R_\theta(\Gamma^-)$ and $\Gamma^+_\theta=R_\theta(\Gamma^+)$, and define $\Gamma_\theta=\Gamma^-_\theta\cup\Gamma^+_\theta$.

Since $m<g'<M$ on $I^\pm$, $\DQ_g(I^\pm)\subset(m,M)$.  Therefore for sufficiently small values of $\theta$, $\Gamma_\theta$ passes the ``vertical line test'', and is thus the graph of a function, which we will call $g_\theta$.  Define $I^-_\theta=\Pi_x(R_\theta(\Gamma^-))$ and $I^+_\theta=\Pi_x(R_\theta(\Gamma^+))$.

Speaking colloquially and for easy reference later, let us color red the points in $\Gamma$ which lie above or below $F^-\cup F^+$ (and think of these points as still colored red after rotation about the origin, that is after application of $R_\theta$).  Then we define $F^-_\theta\subset I^-_\theta$ to be the projection to the $x$-axis of the set of red points in $\Gamma^-_\theta$, and we define $F^+_\theta$ similarly.

If $\theta$ is sufficiently small, then the Items~\ref{item: I^-,I^+ bounded away from origin.}-\ref{item: g(I^-)=g(I^+).} above will still be true when $g$, $I^\pm$, and $F^\pm$ are replaced by $g_\theta$, $I^\pm_\theta$, and $F^\pm_\theta$.  Call this replacement the ``$\theta$-replacement''.  Note that after the $\theta$-replacement, the inequalities in Items~\ref{item: I^-,I^+ bounded away from origin.}-\ref{item: g(I^-)=g(I^+).} are only approximately true, in the sense that $m$ and $M$ may have to be replaced by $m+\epsilon$ and $M-\epsilon$ for some small $\epsilon>0$, but this $\epsilon$ may be chosen arbitrarily small by making $\theta$ sufficiently small.  Thus for sufficiently small $\theta$, the arguments above leading to the conclusion that $0\in\DQ_g(F)$ still work after the $\theta$-replacement.  In particular, the same steps of reasoning allow us to conclude that for some $x^-_\theta\in F^-_\theta$ and $x^+_\theta\in F^+_\theta$, $g_\theta(x^-_\theta)=g_\theta(x^+_\theta)$, and thus $0=\DQ_{g_\theta}(x^-_\theta,x^+_\theta)\in\DQ_{g_\theta}(F_\theta)$.

Let $L_\theta$ denote the horizontal line segment between the two red points $(x^-_\theta,g_\theta(x^-_\theta))$ and $(x^+_\theta,g_\theta(x^+_\theta))$, which lie in the graph of $g_\theta$.  Rotating $L_\theta$ back around the origin through $-\theta$ radians, we obtain a line segment between two red points in the graph of $g$ (that is, between two points in the graph of $g$ which lie above or below $F$).  The slope of this line segment is $\tan(-\theta)$, so we conclude that $\tan(-\theta)\in\DQ_g(F)$.  Since this is true for every sufficiently small $\theta$ (positive or negative), we conclude that $0$ is in the interior of $\DQ_g(F)$.

\end{proof}

\begin{lemma}\label{lem: Porcupine lemma.}
If $g\in C^1([0,1])$ satisfies $\mathscr{C}$, and $F\subset[0,1]$ has positive mass, then for almost every pair $(x_1,x_2)\in (F\times F)\setminus D$, $g'(x_1)\neq \DQ_g(x_1,x_2)\neq g'(x_2)$.
\end{lemma}

\begin{proof}
Suppose by way of contradiction that there is some positive mass set $A\subset (F\times F)\setminus D$ such that for every pair $(x_1,x_2)\in A$, $\DQ_g(x_1,x_2)$ equals either $g'(x_1)$ or $g'(x_2)$.  By dropping to a subset, we may assume that for every $(x_1,x_2)\in A$, $\DQ_g(x_1,x_2)=g'(x_2)$ (or that for every $(x_1,x_2)\in A$, $\DQ_g(x_1,x_2)=g'(x_1)$, in which case we make the appropriate minor changes).  Choose now some fixed $y_0\in\Pi_y(A)$ such that the intersection of $A$ with the horizontal line at height $y_0$ has positive (one dimensional) measure.  Define $$X=\{x\in F:(x,y_0)\in A\}.$$  Then for each $x\in X$, $(x,y_0)\in A$ means that $g'(y_0)=\dfrac{g(y_0)-g(x)}{y_0-x}$, so that $g(x)=g'(y_0)x+(g(y_0)-y_0g'(y_0))$.  That is, $g$ is linear on $X$.  But by choice of $y_0$, $X$ has positive mass, which contradicts the assumption that $f$ satisfies $\mathscr{C}$.
\end{proof}

Proceeding now with the proof of Theorem~\ref{thm: Main.}, Lemma~\ref{lem: Porcupine lemma.} implies that $f'(x_1)\neq\DQ_f(x_1,x_2)\neq f'(x_2)$ for almost every $(x_1,x_2)\in (E\times E)\setminus D$.  Moreover by the Lebesgue density theorem, almost every point in $E$ is a density point of $E$.  Fix now any such point $(x_1,x_2)\in (E\times E)\setminus D$ such that $x_1$ and $x_2$ are density points of $E$, and $f'(x_1)\neq\DQ_f(x_1,x_2)\neq f'(x_2)$.  Let $T_1$ denote the linear map $T_1(x)=ax+b$ such that $T_1(-1)=x_1$ and $T_1(1)=x_2$.  Let $T_2$ denote the linear map $T_2(x)=cx+d$ such that $T_2(-1)=-f(x_1)$ and $T_2(1)=-f(x_2)$.  Define $g(x)=f(T_1(x))+T_2(x)$.  Define $F={T_1}^{-1}(E)$.  This $g$ and $F$ now satisfy the assumptions of Lemma~\ref{lem: Concrete context.}, and thus we conclude that $0=\DQ_g(-1,1)$ is in the interior of $\DQ_g(F)$.

However $\DQ_g(F)$ is obtained from $\DQ_f(E)$ by a dilation (multiplication by ${T_1}'$) followed by a translation (addition of ${T_2}'$).  We conclude that $\DQ_f(x_1,x_2)$ is in the interior of $\DQ_f(E)$.

\section{Limiting Examples to Strengthenings of Theorem~\ref{thm: Main.}.}\label{sect: Limiting counter-example.}

In this section we will display two constructions.  The first construction establishes a limit on the extent to which the assumption on $f$ in part~$1$ of Theorem~\ref{thm: Main.} may be weakened without losing the conclusion.  The second construction shows that no version of part~$2$ of Theorem~\ref{thm: Main.} may be attempted where the difference quotient set of $f$ is replaced by the range of $f$.

\subsection{$\mathscr{B}$ cannot be replaced by $\mathscr{D}$ in Theorem~\ref{thm: Main.} Part~1.}

Our goal is to construct a function $g:[0,1]\to\mathbb{R}$ which is continuously differentiable and satisfies $\mathscr{D}$, and to find a positive mass set contained in $[0,1]$ whose image under $g$ is null.  We will define our function $g$ by first describing a recursive construction of a sequence of nested sets $$[0,1]\times[0,1]\supset G_0\supset G_1\supset\cdots.$$  Having done this, $g$ will be the function whose graph is exactly $\displaystyle\bigcap_{i=0}^\infty G_i$.  We will give a more detailed description of this sequence below, but the reader should first consult Figures~\ref{fig: G_0.}~and~\ref{fig: G_1.} for a conceptual understanding of the recursive construction.  In Figures~\ref{fig: G_0.}~and~\ref{fig: G_1.}, we depict the sets $G_0$ and $G_1$ respectively.  In the recursive step, by which $G_{k+1}$ is produced from $G_k$ (for $k\geq 0$), each rectangle in $G_k$ is replaced by an affine copy of a set very similar to $G_0$ (though somewhat tamped down, as can be seen by the fact that the rectangles from $G_0$ in Figure~\ref{fig: G_0.} are not entirely filled in in Figure~\ref{fig: G_1.}, this ``tamping factor'' will be discussed further later).

\begin{figure}[!ht]
\begin{minipage}[b]{0.45\linewidth}
\centering
\includegraphics[width=\textwidth]{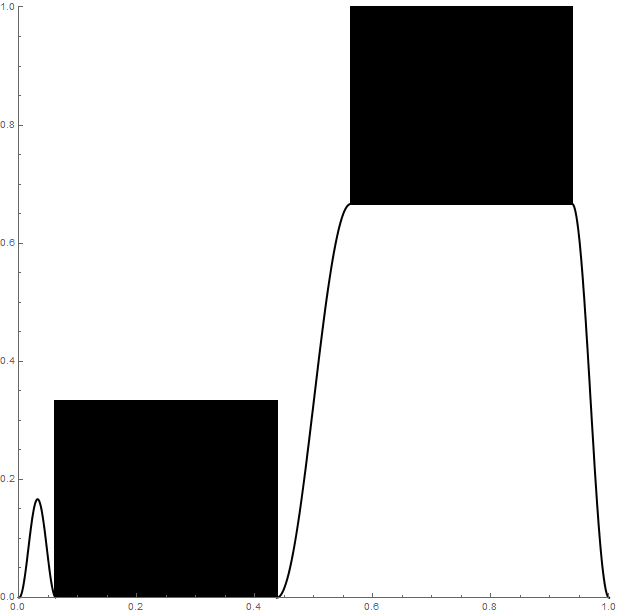}
\caption{The set $G_0$.}
\label{fig: G_0.}
\end{minipage}
\hspace{0.5cm}
\begin{minipage}[b]{0.45\linewidth}
\centering
\includegraphics[width=\textwidth]{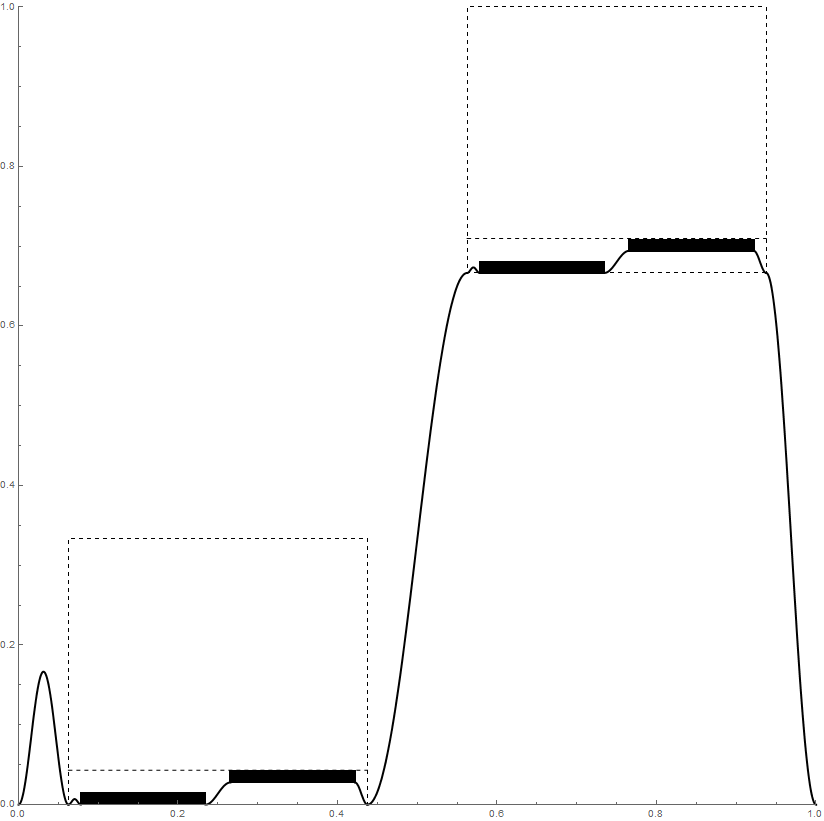}
\caption{The set $G_1$.}
\label{fig: G_1.}
\end{minipage}
\end{figure}

We will now give a more detailed construction.  We will begin with a sequence of auxiliary sets $\{F_k\}$, each contained in $[0,1]\times[0,1]$, and similar in spirit to $G_0$  (see Figure~\ref{fig: F_k.} for the typical $F_k$, these will be the sets which we use to replace the rectangles in $G_k$ in order to form $G_{k+1}$).  Each set in the sequence $\{F_k\}$ will depend on the corresponding element of a certain decreasing sequence of positive real numbers $\{v_k\}$.  The sequence $\{v_k\}$ must be chosen so that each $v_k\in(0,1/6)$, and $v_k\to0$.  Having selected such a sequence $\{v_k\}$, we define $F_k$ to be the union of the closed rectangle $$\left[\dfrac{v_k}{2},\dfrac{1}{2}-\dfrac{v_k}{2}\right]\times\left[0,\dfrac{1}{3}\right],$$ the closed rectangle $$\left[\dfrac{1}{2}+\dfrac{v_k}{2},1-\dfrac{v_k}{2}\right]\times\left[\dfrac{2}{3},1\right],$$ an affine copy of the graph $\{(x,\cos(x)):-\pi\leq x\leq\pi\}$ connecting the two points $(0,0)$ and $\left(\dfrac{v_k}{2},0\right)$, scaled to have total height $1/6$, an affine copy of the graph $\{(x,\sin(x)):-\pi/2\leq x\leq\pi/2\}$ connecting the two points $$\left(\dfrac{1}{2}-\dfrac{v_k}{2},0\right)\text{ and }\left(\dfrac{1}{2}+\dfrac{v_k}{2},\dfrac{2}{3}\right),$$ and an affine copy of the graph $\{(x,\cos(x)):0\leq x\leq\pi\}$ connecting the points $$\left(1-\dfrac{v_k}{2},\dfrac{2}{3}\right)\text{ and }\left(1,0\right).$$

\begin{figure}[!ht]

\centering
\includegraphics[width=.45\textwidth]{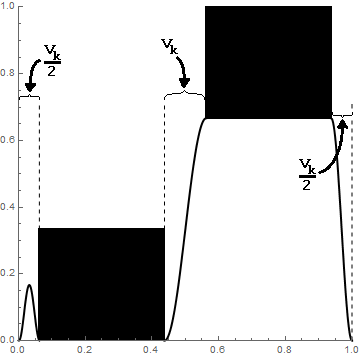}
\caption{The set $F_k$.}
\label{fig: F_k.}

\end{figure}

With the auxiliary sequence of sets $\{F_k\}$ constructed, we will now construct the sequence $\{G_k\}$.  Define $G_0=F_0$.  Choose now another decreasing sequence of positive real numbers $\{h_k\}\subset(0,1)$ which approaches zero (this will be the sequence of ``tamping factors'' mentioned above).  For $k\geq0$, in order to construct $G_{k+1}$ from $G_k$, each one of the closed rectangles in $G_k$ will be replaced with an affine copy of $F_k$.  Let $[a,b]\times[c,d]$ denote some such closed rectangle in $G_k$.  Let $T$ denote the affine map defined by $$T(x,y)=((b-a)x+a,h_k(d-c)y+c),$$ which is a bijection from $[0,1]\times[0,1]$ to $[a,b]\times[c,c+h_k(d-c)]$.  We replace $[a,b]\times[c,d]$ in $G_k$ by $T(F_k)$, and having done this for each closed rectangle in $G_k$, the resulting set we denote by $G_{k+1}$.  We now wish to show that the set $G_\infty=\ds\bigcap G_k$ is the graph of a function $g:[0,1]\to\mathbb{R}$ which has the desired properties.

Let $R_k$ denote the union of the closed rectangles used to form $G_k$ (for any $k\geq0$).  By inspecting the construction, it is easy to see that $R_k$ consists of $2^{k+1}$ identical rectangles, which do not overlap vertically or horizontally (that is, if $R_k$ is projected to the $x$-axis, the shadows of the $2^{k+1}$ rectangles do not overlap, and similarly for the $y$-axis).  Let $s_k$ and $t_k$ denote the width and height respectively of any one of the rectangles in $R_k$.  It follows immediately from the construction of the sequence $\{G_k\}$ that $t_k\to0$, and therefore that $G_\infty$ satisfies the ``vertical line test''.  Note also that elementary compactness considerations imply that for every $x\in[0,1]$, there is some $y\in[0,1]$ such that $(x,y)\in G_\infty$.  That is, $G_\infty$ is the graph of a well defined function on the entire interval $[0,1]$.

In order to show that this function $g$ has the desired properties, we begin by defining $X_k=\Pi_x(R_k)$ and $Y_k=\Pi_y(R_k)$, and $$X_\infty=\displaystyle\bigcap_{k=1}^\infty X_k\text{ and }Y_\infty=\bigcap_{k=1}^\infty Y_k.$$  Note that by the construction, $s_k-2s_{k+1}=2v_k$, so that $$\lambda(X_k\setminus X_{k+1})=2^{k+1}(s_k-2s_{k+1})=2^{k+1}v_k.$$  If we choose $v_k=\dfrac{1}{2^{2k+3}}$, then $\lambda(X_0)=\dfrac{3}{4}$, and we obtain $$\lambda\left(X_\infty\right)=\lambda(X_0)-\ds\sum_{k=0}^\infty\lambda(X_k\setminus X_{k+1})=\dfrac{3}{4}-\ds\sum_{k=0}^\infty 2^{k+1}\cdot\dfrac{1}{2^{2k+3}}=\dfrac{1}{4}>0.$$  On the other hand, each $h_k<1$, so that $t_{k+1}=\dfrac{1}{3}h_kt_k<\dfrac{t_k}{3}$.  It follows that $t_k<\dfrac{1}{3^k}\cdot s_0=\dfrac{1}{3^{k+1}}$, and thus that $\lambda(Y_k)\leq\dfrac{2^{k+1}}{3^{k+1}}$.  Therefore we have $$\lambda(Y_\infty)=\lambda\left(\ds\bigcap_{k=0}^\infty Y_k\right)\leq\lim_{k\to\infty}\lambda(Y_k)\leq\lim_{k\to\infty}\dfrac{2^{k+1}}{3^{k+1}}=0.$$

Therefore since $f(X_\infty)\subset Y_\infty$, we conclude that $\lambda(f(X_\infty))=0$.  This $X_\infty$ is the set which we we are looking for that contradicts the conclusion of second part of Theorem~\ref{thm: Main.}.

The final step in this section is to observe that with an appropriate choice of the tamping factors $\{h_k\}$, $g$ will be continuously differentiable, and to show that $g$ satisfies the property $\mathscr{D}$.


Note that the closure of ${X_k}^c$ is in fact contained in the interior of ${X_{k+1}}^c$.  And on the interior of any ${X_k}^c$, $g$ is continuously differentiable, so that $g$ is continuously differentiable on ${X_\infty}^c$.  On the other hand, by choosing the sequence $\{h_k\}$ to approach zero fast enough, we can ensure that $|g'(x)|<\dfrac{1}{k}$ for each $x\in X_k$.  This in turn implies that for each point $x_0\in X_\infty$, $$\ds\lim_{x\to x_0}g'(x)=0=g'(x_0),$$ so that $g$ is continuously differentiable on $X_\infty$.

Finally, observe that for any $t\in\mathbb{R}$, the horizontal line $y=t$ intersects any given $G_k$ in just one of the rectangles (which has horizontal width $s_k$), and in only finitely many of the smooth curves.  That is, the intersection has linear measure at most $s_k$.  Since the graph of $g$ is contained in any given $G_k$, $\lambda(g^{-1}(t))\leq\ds\lim_{k\to0}s_k=0$.  That is, $f$ satisfies the property $\mathscr{D}$.

In conclusion, we have a function $g\in C^1[0,1]$ which satisfies $\mathscr{D}$ and a positive mass set $X_\infty\subset[0,1]$ such that $g(X_\infty)$ is null.

\subsection{$\DQ_f(E)$ cannot be replaced by $f(E)$ in Theorem~\ref{thm: Main.} Part~2.}

Let $g:[0,1]^n\to\mathbb{R}$ be a measurable function.  We will now construct a set having full density in $[0,1]$, but whose image under $g$ does not contain any interval.  Define $$A:=\{y\in\mathbb{R}:\lambda\left(g^{-1}(y)\right)>0\}.$$  Basic principles of summability imply that $A$ is at most countable, so we may 
find a countable set $B\subset A^c$ which is dense in $\mathbb{R}$.  Therefore by countable additivity, $\lambda\left([0,1]^n\setminus g^{-1}(B)\right)=1$, but $g\left([0,1]^n\setminus 
g^{-1}(B)\right)$ does not intersect $B$, a dense set in $\mathbb{R}$, and thus $g\left([0,1]^n\setminus g^{-1}(B)\right)$ does not contain an interval.

\bibliographystyle{plain}
\bibliography{refs}

\end{document}